\makeatletter

\documentclass[12pt]%
{article}

\newsavebox{\savepar}

\usepackage{amsmath}
\usepackage{amsfonts}
\usepackage{euscript}
\usepackage{oldgerm}
\usepackage{amsthm}
\usepackage{rotating}
\usepackage{mathrsfs}
\usepackage{hyperref}
\usepackage{amssymb}
\usepackage{epsfig}
\usepackage{dsfont}
\usepackage{subfig}
\makeindex

  \newtheorem{theorem}{Theorem}[section]

  \theoremstyle{definition}
\newtheorem{definition}[theorem]{Definition}

  \theoremstyle{remark}
  \newtheorem{remark}[theorem]{Remark}

\theoremstyle{definition}

\theoremstyle{remark}

\begin{document}

\newcommand{\norm}[1]{\left\lVert #1\right\rVert}
\newcommand{\namelistlabel}[1]{\mbox{#1}\hfil}
\newenvironment{namelist}[1]{%
\begin{list}{}
{
\let\makelabel\namelistlabel
\settowidth{\labelwidth}{#1}
\setlength{\leftmargin}{1.1\labelwidth}
}
}{%
\end{list}}

\newcommand{\inp}[2]{\langle {#1} ,\,{#2} \rangle}
\newcommand{\vspan}[1]{{{\rm\,span}\{ #1 \}}}
\newcommand{\R} {{\mathbb{R}}}

\newcommand{\B} {{\mathbb{B}}}
\newcommand{\C} {{\mathbb{C}}}
\newcommand{\N} {{\mathbb{N}}}
\newcommand{\Q} {{\mathbb{Q}}}
\newcommand{\LL} {{\mathbb{L}}}
\newcommand{\Z} {{\mathbb{Z}}}

\newcommand{\BB} {{\mathcal{B}}}

\title{Joint spectrum  of matrix and operator tuples on spaces over bi complex numbers   }
\author{ Akshay S. RANE \footnote{Department of Mathematics, Institute of Chemical Technology, Nathalal Parekh Marg, Matunga, Mumbai 400 019, India, email :  as.rane@ictmumbai.edu.in,} 
\hspace {1mm}
}
\date{ }
\maketitle
\begin{abstract}
In this paper, we generalize the notion of joint eigenvalues and joint spectrum of matrices and operator tupples on a
bi complex Hilbert space. We observe that unlike the spectrum of a  bounded operator on a bi complex Hilbert space is bounded. But this is not the case here.
Even the set of joint eigenvalues of matrix tuples is unbounded. 
\end{abstract}



\newpage
\section{INTRODUCTION}
The set of bi complex numbers is defined as $$ \mathbb{B C}= \lbrace Z =z_{1} +z_{2}j : z_{1},z_{2} \in \mathbb{C}(i) \rbrace $$ where $i$,$j$ are imaginary units. The major work with bi complex numbers is done in \cite{ALSS}. The product of these imaginary units gives us a hyperbolic unit $k$ such that $i.j =k$ and $k^{2}=1.$ The set of bi complex numbers forms a ring. Bicomplex numbers are just like quaternions which is the generalization of complex numbers by means of entities specified by four real numbers. But the difference is, quaternions is a non commutative ring and it is a division algebra where as set of bicomplex numbers is commutative ring which is not division algebra. Since we have two imaginary units and 1 hyperbolic unit, we have three conjugates.
\begin{enumerate}
	\item $\overline{Z}= \overline{z_{1}}+\overline{z_{2}}j$  (\textbf{bar-conjugation}).
	\item $Z^{\dagger}= z_{1}-z_{2}j$  (\textbf{$\dagger$ - conjugation}).
	\item $Z^{*}= \overline{z_{1}}-\overline{z_{2}}j$  (\textbf{*- conjugation}).
\end{enumerate}a
So the set of bi complex numbers have zero divisors. There are two very special zero divisors $e_{1}$ and $e_{2}$ which are $$ e_{1} =\frac{1+k}{2}$$ and $$e_{2}=\frac{1-k}{2}.$$ Using these zero divisors any bi complex number $Z$ can be expressed as $$ Z= \beta_{1} e_{1} + \beta_{2}e_{2}$$ where $\beta_{1} =z_{1}-z_{2}i$ and $\beta_{2} = z_{1} + z_{2} i.$ 
\noindent
So this gives us the following Idempotent decomposition of set of b icomplex numbers :
$$ \mathbb{ B C} = \mathbb{C}(i) {e_{1}} + \mathbb{C}(i){e_{2}}. $$
 We now define the set of hyperbolic numbers.
The set $\mathbb{D}$ of hyperbolic numbers is defined as $$ \mathbb{D} = \lbrace h = h_{1} + k h_{2} | h_{1},h_{2} \in \mathbb{R} \rbrace.$$
The set of all positive hyperbolic numbers is defined as $$ \mathbb{D^{+}} =\lbrace h =e_{1}a_{1}+ e_{2}a_{2} | a_{1},a_{2} \geq 0 \rbrace.$$
For $\alpha_{1} ,\alpha_{2} \in \mathbb{D}$, we have $ \alpha_{1} < \alpha_{2}$ whenever $\alpha_{2}-\alpha_{1} \in \mathbb{D^{+}}.$ 
This relation is reflexive, transitive, antisymmetric and so it defines partial order.
For bi complex numbers we have two norms , first is Euclidean norm $$ |Z| = \sqrt{x_{1}^{2} + x_{2}^{2} + y_{1}^{2} + y_{2}^{2}} = \sqrt{|z_{1}|^{2} + |z_{2}|^{2}}.$$ The second is hyperbolic norm or $\mathbb{D}$ valued norm.The map $$ |.|_{k} : \mathbb{B C} \rightarrow \mathbb{D^{+}} $$ with following properties : 
\begin{itemize}
	\item $|Z|_{k} =0$ if and only if Z=0.
	\item $|Z.W|_{k} = |Z|_{k} .|W|_{k}$ for any $Z, W \in \mathbb{ B C}.$
	\item $|Z +W|_{k} < |Z|_{k} +|W|_{k}.$
\end{itemize} 
A bi complex module X is said to be $F -\mathbb{B C}$ module if X is complete hyperbolic normed linear space.
Now we define $\mathbb{B C}$ linear operator. Let X be a $\mathbb{B C}$ module. Let $T : X \rightarrow X$ be a map. It is said to be $\mathbb{ B C}$ linear operator if 
\begin{itemize}
	\item $T(x + y)= T(x) + T(y)$
	\item $T(\alpha x)= \alpha T(x)$
\end{itemize}
for all $x,y \in X$ and for all $\alpha \in \mathbb{B C}.$The bi complex inner product is defined in \cite{ALSS}.
Let $X$ be $\mathbb{B} \mathbb{C}$ module. A map $$ \langle.,. \rangle : X \times X \rightarrow \mathbb{B} \mathbb{C}$$ is said to be a $\mathbb{B} \mathbb{C}$ inner product if it satisfies following properties for all $x,y,z \in X$ and $\alpha \in \mathbb{B} \mathbb{C}$:
\begin{enumerate}
	\item $\langle x,y+z \rangle $ = $\langle x,y \rangle  + \langle x+z \rangle $
	\item $\langle \alpha x,y \rangle $ = $\alpha \langle x,y \rangle  $
	\item $\langle x,y \rangle $= $\langle y,x \rangle ^{*}$
	\item $\langle x,x \rangle  \in \mathbb{D^{+}}$ and $\langle x,x \rangle =0$ if and only if $x=0.$
\end{enumerate}
A $ \mathbb{B C}$ inner product module $X$ is said to be a bicomplex Hilbert module if $X$ is complete with respect to the hyperbolic norm generated by the inner product . 
\begin{definition}
	Let $H_1$ and $H_2$ be two bi complex Hilbert spaces. Then the bicomplex adjoint operator $T^*: H_2 \rightarrow H_1$ for a bounded operator $$T^*: H_2 \rightarrow H_1$$ is defined by
	the equality $$ \langle Tx,y \rangle = \langle x , T^* y \rangle .$$ 
\end{definition}
Note the Bicomplex adjoint $T^*$ can also be represented by $$ T^* = e_1 T_1^* + e_2 T_2^*. $$ The notion of eigen values, spectrum of a bounded linear operators in context of modules over bi complex numbers is found in \cite{ALSS} and \cite{CSS}. The functional calculus is also developed here. It is observed that unlike the usual spectral theory over complex numbers, in this case the spectrum of a bounded linear operator is not bounded subset of $\mathbb{C}.$ The idea of joint spectrum of operator tuple was introduced in Taylor. The spectrum is non empty provided the the operator tuple commutes This definition involves the idea of Kozul complex. Vascilescu \cite{Vas}.  gave a characterization of the spectrum of commuting pair of operator tuple similar to that of the usual operator case. In this case also the joint spectrum is a bounded subset of cartesian product of copies of $\mathbb{C}.$  The joint eigenvalues and the spectral radius formula for matrix tuple was dealt in Bhatia \cite{Bhatia}. In this paper we propose the definition of joint eigenvalue of matrix tuple. We observe that the joint eigenvalues of matrix tuple is not a bounded subset of $\mathbb{BC}.$ We also generalize the paper of Vascilescu \cite{Vas} and propose the definition of joint spectrum of a pair of commuting operator tuple. We define the notion of joint approximate spectrum, residual spectrum and
joint point spectrum of a pair of operator tuple. These quantities were defined in Dash \cite{Dash} for usual pair of operator tuples. We observe that the joint approximate spectrum is bounded whereas the other parts are un bounded. 
\section{Joint Eigenvalue of matrices} 
  Suppose $A_1, A_2,\ldots, A_m $ are $d \times d$ matrices with $ \mathbb{B}\mathbb{C}$ entries. Then a bicomplex tupple  $(\lambda_1, \lambda_2, \ldots, \lambda_m)$ is said to be a joint eigenvalue of $(A_1, A_2,\ldots, A_m )$ if there exist a common non zero eigenvector $x \in \mathbb{B} \mathbb{C}^d$ such that $$ A_i x= \lambda_i x$$ for $i=1,\ldots,m.$ Using the idempotent decomposition of $$ A_i= A_i^ \prime e_1 + A_i^{ \prime \prime}e_2 $$ and $$ x= x_1 e_1 + x_2 e_2,\; \lambda_i= \mu_i e_1 + \gamma_i e_2 $$ we have 
  the following relations $$  A_i^ \prime x_1 = \mu_i  x_1,\;  A_i^ {\prime \prime} x_2 = \gamma_i  x_2, $$
  for $i=1,\ldots,m $ and $ A_i^\prime, A_i^{\prime \prime}$ are complex matrices. Since $x\neq 0$, this implies $x_1 \neq0$ or $ x_2 \neq 0$. Suppose $x_1 \neq 0.$ Setting $x_2=0 $ any $ \gamma_i \in \mathbb{C}$ can be a joint eigenvalue of $(A_1,A_2,\ldots,A_m)$ So we have the following 
  \begin{theorem}
  	$$ \sigma_p(A_1,A_2,\ldots,A_m) = \sigma_p(A_1^\prime,A_2^\prime,\ldots,A_m^\prime)e_1+ \mathbb{C}e_2 \bigcup  \mathbb{C}e_1+  \sigma_p(A_1^{\prime \prime},A_2^{\prime\prime},\ldots,A_m^{\prime\prime})e_2,$$ where
  	$\sigma_p(A_1,A_2,\ldots,A_m)$ denotes the set of joint eigenvalues of $(A_1,A_2,\ldots,A_m)$
  	\end{theorem}
  \begin{remark}
  The matrix tupple is commuting if $A_i A_j= A_j A_i$ for $i,j=1,\ldots,m.$ It can be easily noted that if $(A_1,A_2,\ldots,A_m)$ is commuting matrix tupple with bi complex entries iff 
 $ (A_1^\prime,A_2^\prime,\ldots,A_m^\prime)$ and $(A_1^{\prime \prime},A_2^{\prime\prime},\ldots,A_m^{\prime\prime})$ are commuting pairs of matrix tupples with complex entries. 
 \end{remark}
\begin{theorem}
		$ \sigma_p(A_1,A_2,\ldots,A_m) \neq \phi $ if $ (A_1,A_2,\ldots, A_m) $ is a commuting pair of matrices with bi complex entries.
\end{theorem}
\begin{proof}
	The proof follows from the above remark, the above theorem  and the fact that if $(C_1, C_2,\ldots, C_m)$ is a a commuting tuple of matrices with complex entries than the spectrum is non empty.
\end{proof}
\begin{remark}
	The joint point spectrum of $ (A_1,A_2,\ldots,A_m)$ is unbounded.
\end{remark}
There fore it makes to talk about the restricted spectrum.
\begin{definition}
	The restricted spectrum can be defined as 
		$$ \sigma_{r,p}(A_1,A_2,\ldots,A_m) = \sigma_p(A_1^\prime,A_2^\prime,\ldots,A_m^\prime)e_1+ \sigma_p(A_1^{\prime \prime},A_2^{\prime\prime},\ldots,A_m^{\prime\prime})e_2,$$ 
\end{definition}
Since $A_1^\prime,A_2^\prime,\ldots,A_m^\prime$ and $ A_1^{\prime \prime},A_2^{\prime\prime},\ldots,A_m^{\prime\prime}$ there are unitary matrices $U^\prime$ and $ U^{\prime \prime}$ such that 
$$  U^\prime A_j^\prime (U^{\prime})^*$$ and $$  U^{\prime \prime }A_j^{\prime \prime} {U^{\prime \prime}}^*$$ are upper triangular matrices. 
Defining $$ V = e_1 U^\prime + e_2 U^{\prime \prime} $$ which is a unitary matrix and $VA_jV^*$ for $j=1,\ldots,m$  is an upper triangular bi complex matrix with diagonal elements of $\sigma_{r,p}(A_1,A_2,\ldots,A_m)= \left \lbrace (\lambda_i^1, \ldots, \lambda_i^m ) | i=1,\ldots,d \right \rbrace,\;k=1,\ldots,m. $ 
Define for $\lambda_i=  (\lambda_i^1, \ldots, \lambda_i^m ),$ The norm $$ \|\lambda_i\|_p ^2= \frac{ (\displaystyle \sum_{k=1}^m |\mu_i^k|^p)^{1/p} + ( \displaystyle \sum_{k=1}^m |\gamma_i^k|^p)^{1/p}}{2}$$
where $$ \lambda_i^k = \mu_i^k e_1 + \gamma_i^k e_2 $$
The Geometric spectral radius of the tuple $T=(A_1, A_2, \ldots, A_m)=e_1T_1 + e_2 T_2$ is defined as 
$$ r_p(T) = \max \lbrace \|\lambda\|_p : \lambda \in \sigma_{r,p}(A_1,A_2,\ldots,A_m) \rbrace$$
Certainly we have $$ \|\mu_i\|_p \leq r_p(T_1) \;and \; \|\gamma_i\|_p \leq r_p(T_2)$$
So that 
$$ \|\lambda_i\|_p ^2\leq \frac{r_p(T_1) + r_p(T_2)}{2} $$
As a result
$$ \|\lambda_i \|_p \leq \sqrt{ \frac{r_p(T_1) + r_p(T_2)}{2} }\leq\sqrt{\frac{ \|T_1\|_p + \|T_2\|_p}{2}} =\|T\|_p$$
and hence 
$$ r_p(T) \leq \|T\|_p,$$
where $$\|T_1\|_p= \sup_{\|x_1\|_p=1} (\sum_{j=1}^m \|A_j^\prime(x_1)\|_p)^\frac{1}{p}$$
and $$\|T_2\|_p= \sup_{\|x_2\|_p=1} (\sum_{j=1}^m \|A_j^{\prime\prime}(x_2)\|_p)^\frac{1}{p}$$
\section{Joint spectrum of a pair of operator tuple}
To talk about that $ z $ is in the resolvent set of $T$ a bounded linear operator on a Hilbert space $H$, we require that $T-zI$ is invertible and bounded linear.
This implies that the short exact sequence 
$$ {0} \rightarrow H \rightarrow H \rightarrow {0} $$ is exact. That is ${0}=ker(T-zI)$ and $Im(T-zI)=H $. Like wise a bicomplex number $z=e_1 z_1 + e_2 z_2$ is in the resolvent of $T=e_1T_1 + e_2 T_2$ if the two short exact sequences $$ {0} \rightarrow e_1H \rightarrow e_1H \rightarrow {0} $$ and $$  {0} \rightarrow e_2H \rightarrow e_2H \rightarrow {0} $$

Now we want to extend the notion of spectrum of joint spectrum of an operator tupple $(T_1, T_2)$ where $T_1$ and $T_2$ are bounded operators which commute on a bi-complex Hilbert space $H$.
We know that the operators $T_1$ and $T_2$ can be written in its idempotent decomposition. That is 
$$ T_1= e_1 T_1^\prime + e_2 T_1^{\prime\prime} $$
and
$$ T_2= e_1 T_2^\prime + e_2 T_2^{\prime\prime} $$
Consider a matrix of operators
$\begin{bmatrix}
	T_1 & T_2 \\
	-T_2^*& T_1^*
\end{bmatrix}$ The above matrix of operators can be decomposed as 
$\begin{bmatrix}
T_1 & T_2 \\
-T_2^*& T_1^*
\end{bmatrix}= 
e_1 \begin{bmatrix}
T_1^\prime  & T_2^\prime  \\
-T_2^{\prime *}& T_1^{\prime *}
\end{bmatrix}
+ e_2 \begin{bmatrix}
T_1^{\prime \prime}  & T_2^{\prime \prime} \\
-T_2^{\prime \prime  *}& T_1^{\prime \prime *}
\end{bmatrix}.
$\\
The second and the third matrix of operators act on $ H \oplus H$ and more specifically on $ e_1H \oplus e_1H$ and $ e_2H \oplus e_2H$ respectively.

Let $(z_1,z_2) \in \mathbb{BC} \times \mathbb{BC}.$ The idempotent decomposition of $z_1 = z_1^\prime e_1 + z_1^{\prime \prime} e_2 $ and 
$z_2 = z_2^\prime e_1 + z_2^{\prime \prime} e_2.$\\
$\begin{bmatrix}
z_1I -T_1 & z_2I- T_2 \\
-z_2^*I + T_2^*&z_1^*I-T_1^*
\end{bmatrix}= 
e_1 \begin{bmatrix}
z_1^\prime I- T_1^\prime  & z_2^\prime I- T_2^\prime  \\
-z_2^{\prime *}I  +T_2^{\prime *}& z_1^{\prime *} I -T_1^{\prime *}
\end{bmatrix}
+ e_2 \begin{bmatrix}
z_1^{\prime \prime} I - T_1^{\prime \prime}  & z_2^{\prime \prime}I -T_2^{\prime \prime} \\
-z_2^{\prime \prime*} I+T_2^{\prime \prime  *}& z_1^{\prime \prime*}I- T_1^{\prime \prime *}
\end{bmatrix}
$
Following corollary 2.2 of \cite{Vas} we propose the following definition of the spectrum of a commuting operator tuple $(T_1, T_2)$ . Note that
using the above decomposition, the matrix $\begin{bmatrix}
z_1I -T_1 & z_2I- T_2 \\
-z_2^*I + T_2^*&z_1^*I-T_1^*
\end{bmatrix} $ is invertible iff the matrices of operators $ \begin{bmatrix}
z_1^\prime I- T_1^\prime  & z_2^\prime I- T_2^\prime  \\
-z_2^{\prime *}I  +T_2^{\prime *}& z_1^{\prime *} I -T_1^{\prime *}
\end{bmatrix}$  and $ 
 \begin{bmatrix}
	z_1^{\prime \prime} I - T_1^{\prime \prime}  & z_2^{\prime \prime}I -T_2^{\prime \prime} \\
	-z_2^{\prime \prime*} I+T_2^{\prime \prime  *}& z_1^{\prime \prime*}I- T_1^{\prime \prime *}
\end{bmatrix}
$\\
This implies that the  matrix $\begin{bmatrix}
z_1I -T_1 & z_2I- T_2 \\
-z_2^*I + T_2^*&z_1^*I-T_1^*
\end{bmatrix} $ is not invertible  iff atleast one of the matrices of operators $ \begin{bmatrix}
z_1^\prime I- T_1^\prime  & z_2^\prime I- T_2^\prime  \\
-z_2^{\prime *}I  +T_2^{\prime *}& z_1^{\prime *} I -T_1^{\prime *}
\end{bmatrix}$  and $ 
\begin{bmatrix}
z_1^{\prime \prime} I - T_1^{\prime \prime}  & z_2^{\prime \prime}I -T_2^{\prime \prime} \\
-z_2^{\prime \prime*} I+T_2^{\prime \prime  *}& z_1^{\prime \prime*}I- T_1^{\prime \prime *}
\end{bmatrix}
$ is not invertible. So we have the following.
\begin{definition}
T $ (\lambda_1 , \lambda_2) $ is in the joint spectrum of $(T_1,T_2)$ iff atleast one of the operator matrices
 $ \begin{bmatrix}
\lambda_1^\prime I- T_1^\prime  & \lambda_2^\prime I- T_2^\prime  \\
-\lambda_2^{\prime *}I  +T_2^{\prime *}& \lambda_1^{\prime *} I -T_1^{\prime *}
\end{bmatrix}$  and $ 
\begin{bmatrix}
\lambda_1^{\prime \prime} I - T_1^{\prime \prime}  & \lambda_2^{\prime \prime}I -T_2^{\prime \prime} \\
-\lambda_2^{\prime \prime*} I+T_2^{\prime \prime  *}& \lambda_1^{\prime \prime*}I- T_1^{\prime \prime *}
\end{bmatrix}
$ is not invertible.
\end{definition}
Thus $ (\lambda_1 , \lambda_2) $ is in the joint spectrum of $(T_1,T_2)$ iff  $ (\lambda_1^\prime  , \lambda_2^\prime ) $ is in the joint spectrum of $(T_1^\prime,T_2^\prime)$ or 
$ (\lambda_1^{\prime\prime} , \lambda_2^{\prime \prime} ) $ is in the joint spectrum of $(T_1^{\prime \prime},T_2^{\prime \prime}).$
Here we assume that $H$ is a Hilbert $\mathbb{BC}$ module endowed with the norm $\|.\|$ . Here we consider the norm $$ \|.\|= \frac{1}{\sqrt{2}}( \|.\|_{H_1}^2+ \|.\|_{H_2}^2)$$ where the
$\|.\|_{H_1}$ and $\|.\|_{H_2}$ are induced from the inner products on $H_1$ and $H_2$ treated as complex inner product spaces. It is a real norm.
The pair $(\lambda_1, \lambda_2)$ is in the approximate point spectrum of $(T_1,T_2)$ iff there exist sequence of unit vectors $\|x_n\|$ such that 
$$\|(T_1- \lambda_1 I)x_n\|, \; \|(T_2- \lambda_2 I)x_n\|$$ both tend to $0.$
$\|x_n=x_n^\prime e_1+ x_n^{\prime \prime}e_2\|=1 $. It could happen that  $ \|x_n^\prime\|_{H_1}=\sqrt{2}$ and   $ \|x_n^{\prime \prime}\|_{H_2}=0$ 
Define $y_n = \frac {x_n^\prime }{\sqrt{2}},\; \|y_n\|_{H_1}=1$.
$$\|(T_1- \lambda_1 I)x_n\| \rightarrow 0,\;\|(T_2- \lambda_2 I)x_n\|\rightarrow 0$$ implies 
$$\|(T_1^\prime - \lambda_1^\prime  I)y_n^\prime \|_{H_1} \rightarrow 0$$ and 
$$\|(T_1^{\prime \prime} - \lambda_1^{\prime \prime}  I)x_n^{\prime \prime}\|_{H_2} \rightarrow 0$$ 
Since $x_n^{\prime \prime}=0$, $\lambda_1^{\prime\prime}$ can be any complex number. Also 
$$\|(T_2^\prime - \lambda_2^\prime  I)y_n^\prime \|_{H_1} \rightarrow 0$$ and 
$$\|(T_2^{\prime \prime} - \lambda_2^{\prime \prime}  I)x_n^{\prime \prime}\|_{H_2} \rightarrow 0.$$ 
By the same reason, $\lambda_2^{\prime \prime}$ can be anything.
The above argument shows that $ \sigma_{ap}(T_1,T_2) \subset \sigma_{ap}(T_1^\prime, T_2^\prime) e_1 +( \mathbb{C} \times \mathbb{C}) e_2 \bigcup ( \mathbb{C} \times \mathbb{C}) e_1+ \sigma_{ap}(T_1^{\prime \prime}, T_2^{\prime \prime})e_2 .$
As observed before, in the case of matrices, we have similar results from Dash \cite{Dash}.
\begin{theorem}
	If $(T_1,T_2)$ is a commuting pair of bounded operator tupple, 
	$$ \sigma_p(T_1,T_2) = \sigma_p(T_1^\prime,T_2^\prime)e_1+ ( \mathbb{C} \times \mathbb{C})e_2 \bigcup  ( \mathbb{C} \times \mathbb{C})e_1+  \sigma_p(T_1^{\prime \prime},T_2^{\prime\prime})e_2,$$ where
	$\sigma_p(T_1,T_2)$ denotes the set of joint eigenvalues of $(T_1,T_2)$
\end{theorem}
\begin{theorem}
	If $(T_1,T_2)$ is a commuting pair of bounded operator tupple, 
	$$ \sigma_r(T_1,T_2) = \overline{ \sigma_p((T_1^\prime)^*,(T_2^\prime)^*)}e_1+ ( \mathbb{C} \times \mathbb{C})e_2 \bigcup  ( \mathbb{C} \times \mathbb{C})e_1+  \overline{\sigma_p((T_1^{\prime \prime})^*,(T_2^{\prime\prime})^*)}e_2,$$ where
	$\sigma_r(T_1,T_2)$ denotes the set of residual spectrum  of $(T_1,T_2)$
\end{theorem}

\end{document}